\documentclass[10pt]{amsart}
\usepackage{tikz,array, verbatim}
\usepackage{amsfonts, amsmath, latexsym, epsfig, caption}
\usepackage{amssymb, color}

\usepackage{epsf}
\usepackage{array}
\usepackage{ragged2e}
\usepackage{hyperref}
\usepackage{textgreek}
\title[On the Voronoi Conjecture for combinatorially Voronoi parallelohedra]{On the Voronoi Conjecture for combinatorially Voronoi parallelohedra in dimension five}

\def\QuotS#1#2{\leavevmode\kern-.0em\raise.2ex\hbox{$#1$}\kern-.1em/\kern-.1em\lower.25ex\hbox{$#2$}}

\begin{document}

\author[M. D. Sikiri\'c]{Mathieu Dutour Sikiri\'c}
\address{Mathieu Dutour Sikiri\'c, Rudjer Boskovi\'c Institute, Bijeni\v{c}ka 54, 10000 Zagreb, Croatia}
\email{mathieu.dutour@gmail.com}

\author[A. Garber]{Alexey Garber}
\address{Alexey Garber, School of Mathematical \& Statistical Sciences, The University of Texas Rio Grande Valley, 1 West University Blvd, Brownsville, TX, 78520, USA}
\email{alexeygarber@gmail.com}

\author[A. Magazinov]{Alexander Magazinov}
\address{Alexander Magazinov, Faculty of Mathematics, ``Higher School of
Economics'' National Research University, 6 Usacheva str., Moscow 119048,
Russia}
\email{amagazinov@hse.ru}


\newcommand{\RR}{\ensuremath{\mathbb{R}}}
\newcommand{\NN}{\ensuremath{\mathbb{N}}}
\newcommand{\QQ}{\ensuremath{\mathbb{Q}}}
\newcommand{\CC}{\ensuremath{\mathbb{C}}}
\newcommand{\ZZ}{\ensuremath{\mathbb{Z}}}
\newcommand{\TT}{\ensuremath{\mathbb{T}}}

\newtheorem{theorem}{Theorem}[section]
\newtheorem{proposition}[theorem]{Proposition}
\newtheorem{corollary}[theorem]{Corollary}
\newtheorem{lemma}[theorem]{Lemma}
\newtheorem{problem}[theorem]{Problem}
\newtheorem*{conjecture}{Conjecture}
\newtheorem{question}{Question}
\newtheorem{claim}{Claim}

\theoremstyle{definition}
\newtheorem{definition}[theorem]{Definition}

\theoremstyle{remark}
\newtheorem*{remark}{Remark}

\begin{abstract}
In a recent paper Garber, Gavrilyuk and Magazinov proposed a sufficient
combinatorial condition for a parallelohedron to be affinely Voronoi. We show that 
this condition holds for all five-dimensional Voronoi parallelohedra.
Consequently, the Voronoi conjecture in $\RR^5$ holds if and only if
every five-dimensional parallelohedron is combinatorially Voronoi.
Here, by saying that a parallelohedron $P$ is {\it combinatorially 
Voronoi}, we mean that the tiling $\mathcal T(P)$ by translates of
$P$ is combinatorially isomorphic to some tiling $\mathcal T(P')$,
where $P'$ is a Voronoi parallelohedron, and that the isomorphism
naturally induces a linear isomorphism of lattices $\Lambda(P)$ and
$\Lambda(P')$.

We also propose a new sufficient condition implying that
a parallelohedron is affinely Voronoi. The condition is
based on the new notion of the Venkov complex associated with a
parallelohedron.
\end{abstract}

\maketitle

\section{Introduction}

A convex $d$-dimensional polytope $P$ is called a {\it parallelohedron} if it tiles $\RR^d$ by translations only. Consider the family of all tilings 
by translates of a given parallelohedron $P$. A remarkable result
by Venkov~\cite{Venkov} and, independently, McMullen~\cite{McMullen}
asserts that this family necessarily contains a face-to-face tiling. 
Without loss of generality, we can impose a further requirement that
the face-to-face tiling contains $P$ as one of its tiles. We denote
such a tiling by $\mathcal T(P)$. 

Minkowski~\cite{Minkowski} established that all parallelohedra are
centrally symmetric. Consequently, the centers of all tiles of 
$\mathcal T(P)$ form a lattice provided the origin $\mathbf 0$ is the center of $P$. We denote this lattice by $\Lambda(P)$.
Then the set of tiles of
$\mathcal T(P)$ is exactly $\{ P + t \, \vert \, t \in \Lambda(P) \}$.

It is easy to check that parallelograms and centrally symmetric hexagons 
are parallelohedra for $d = 2$, and that there are no other 
two-dimensional parallelohedra. The full classification of parallohedra 
of a given dimension $d$ exists only for $d \leq 4$. The case $d = 3$ is
due to Fedorov \cite{fedorov-1885}, and the case $d = 4$ is due to 
Delaunay~\cite{Del} with a correction by Stogrin~\cite{Stogrin_triclinic}.

When speaking about classification of parallelohedra, it is
necessary to specify the equivalence relation used to split parallelohedra in equivalence classes. We introduce 
the notion of equivalence in the following definition, where
$\mathcal F(\mathcal T(P))$ denotes the poset of all faces of 
$\mathcal T(P)$ ordered by inclusion.

\begin{definition}\label{def:equiv}
Two $d$-dimensional parallelohedra, $P$ and $P'$, are {\it equivalent} if
there is an isomorphism of face posets 
$f : \mathcal F(\mathcal T(P)) \to \mathcal F(\mathcal T(P'))$
inducing a linear isomorphism of lattices 
$f_* : \Lambda(P) \to \Lambda(P')$ by restricting the action of $f$ 
to $d$-dimensional tiles and then passing to their centers.
\end{definition}

\begin{remark}
Equivalence (in the sense of Definition~\ref{def:equiv}) for Voronoi
parallelohedra reduces to the notion of an {\it L-type}.
More precisely, two Voronoi parallelohedra are equivalent if and only if
they belong to the same L-type. The concept of L-types originates
in the work of Voronoi~\cite{VoronoiII}. See 
also~\cite[Section 3]{ClassifDim5} for a modern treatment.
\end{remark}

Of course, if $P$ and $P'$ are equivalent as parallelohedra, then they
are combinatorially equivalent as convex polytopes. In addition, 
the equivalence classes also retain information about the 
lattice structure of the tiling, in particular, about central
symmetries preserving the lattice. However,
the authors are not aware of an example of two non-equivalent
parallelohedra that are combinatorially equivalent as convex polytopes.

One of the most famous conjectures in the theory of parallelohedra is
stated by Voronoi~\cite{VoronoiII} and reads as follows.
\begin{conjecture}
Every parallelohedron $P$ is affinely Voronoi.
\end{conjecture}

Here we call a $d$-dimensional parallelohedron {\it Voronoi} if
it is a Dirichlet-Voronoi polytope for some $d$-dimensional lattice,
and {\it affinely Voronoi} if it is an affine image of a Voronoi
parallelohedron. 

\begin{definition}
We say that a parallelohedron $P$ is {\it combinatorially Voronoi}, if there is a Voronoi parallelohedron $P'$ such that $P$ and $P'$ are equivalent in the sense of Definition \ref{def:equiv}.

Note that this equivalence is stronger than just combinatorial equivalence because we require for a linear isomorphism between the corresponding lattices be induced by combinatorial bijection.
\end{definition}

The Voronoi conjecture has been proved for some families of parallelohedra with certain combinatorial restrictions on their properties, see \cite{VoronoiII,Zhitomirskii,Ordine_thesis,GGM,Grish} for details, but still remains open in general. The conjecture was also checked for all 
$3$- and $4$-dimensional parallelohedra, but for $d\geq 5$ remains open 
as well.

Recently, Garber, Gavrilyuk and Magazinov~\cite{GGM} proposed a 
sufficient condition for a parallelohedron to be affinely Voronoi. 
In this paper we combine this condition with the complete classification 
five-dimensional Voronoi parallelohedra, due to 
Dutour Sikiri\'c, Garber, Sch\"urmann, and Waldmann \cite{ClassifDim5},
in order to prove the following main result.

\begin{theorem}\label{thm:main}
A 5-dimensional parallelohedron is affinely Voronoi if and
only if it is equivalent to a Voronoi parallelohedron.
\end{theorem}

Theorem~\ref{thm:main} essentially reduces the Voronoi conjecture in
5 dimensions to its weaker, combinatorial, version. 

Earlier, Garber~\cite{Garber_4dim} verified the condition of~\cite{GGM}
for all 4-dimensional parallelohedra.

Additionaly, we propose yet another sufficient condition, also depending
only on the equiavlence class, implying that a parallelohedron
is affinely Voronoi. This condition generalizes both~\cite{GGM} 
and~\cite{Ordine_thesis}.

The paper is organized as follows. 

In Section~\ref{sec:pisurface}
the key concepts and statements of the paper~\cite{GGM} are reproduced.

In Section~\ref{sec:simplicial} we introduce the {\it Venkov complex},
a 2-dimensional simplicial complex $Ven(P)$ associated with
the combinatorics of $P$. The notion of the Venkov complex is then 
used to formulate a sufficient condition 
(Theorem~\ref{thm:simp-condition}) for a parallelohedron to be affinely 
Voronoi. 

In Section~\ref{sec:graph} the condition of~\cite{GGM} is
reformulated in a discrete form, namely, in terms of the 1-skeleton of
$Ven(P)$  (or, equivalently, the {\it red Venkov graph} $VG_r(P)$). 

In Section~\ref{sec:computational} we describe the algorithm
to verify Theorem~\ref{thm:main} and provide the details of
implementation. 

In Section~\ref{sec:reduction} we show that 
the condition of Theorem~\ref{thm:simp-condition} implies both the 
conditions of~\cite{GGM} and~\cite{Ordine_thesis}. 

Finally, in Section~\ref{sec:conclusion} some open questions are proposed.

\section{The $\pi$-surface of a parallelohedron}\label{sec:pisurface}

We follow the exposition of \cite{GGM} in order to formulate the sufficient condition for Voronoi conjecture and refer to this paper for more details.

We fix a $d$-dimensional parallelohedron $P$. As before, let $\mathcal{T}(P)$ be the face-to-face tiling of $\RR^d$ with translations of $P$. Each 
$(d-2)$-face of the polytope $P$ can be attributed to one of two types
according to the following well-known proposition.

\begin{proposition}[See, for instance,~\cite{Venkov} or~\cite{McMullen}]
\label{prop:d-2-faces}
Each $(d-2)$-face $F$ of $P$ is incident either exactly to 3 
or exactly to 4 copies of $P$ in the tiling $\mathcal T(P)$.
\end{proposition}

\begin{remark}
Proposition~\ref{prop:d-2-faces} is equivalent to the so-called
Minkowski--Venkov condition on belts of parallelohedra.
\end{remark}

\begin{definition}
If a $(d-2)$-face $F$ is incident exactly to three copies of $P$ in 
$\mathcal{T}(P)$, then $F$ is called {\em primitive}. Otherwise 
$F$ is called {\em non-primitive}.
\end{definition}

With a tiling $\mathcal T(P)$ we associate the {\it dual complex} 
$\mathcal D(P)$ according to the following definition.

\begin{definition}
Let $G$ be a face of $\mathcal T(P)$ of an arbitrary dimension 
$0 \leq k \leq d$. Denote by $D(G)$ the {\it dual $(d - k)$-cell}
of $G$, which is, by definition, the centers of all translates
of $P$ in $\mathcal T(P)$ that contain $G$ as a face. The poset
of all dual cells for the faces of $\mathcal T(P)$ with ordering
by inclusion is called the {\it dual complex} of $\mathcal T(P)$
and denoted by $\mathcal D(P)$.
\end{definition}

\begin{remark}
$\mathcal D(P)$ is dual to the poset of faces of $\mathcal T(P)$,
since the map $G \mapsto D(G)$ inverts the relation of inclusion. 
\end{remark}

\begin{remark}
We equip each dual cell $D(G)$ with a face structure by considering
its subcells, i.e., all dual cells $D(G') \subseteq D(G)$, where
$G' \supseteq G$. Therefore we are able to talk about the combinatorics
of $D(G)$.
\end{remark}

A classification of possible combinatorial types of dual 3-cells
is available due to a classical result of Delaunay.

\begin{proposition}[See~\cite{Del}]
Let $G$ be a $(d - 3)$-dimensional face of $\mathcal T(P)$. Then
\begin{enumerate}
\item The dual 3-cell $D(G)$ is combinatorially equivalent to one of the
following five 3-polytopes: a tetrahedron, an octahedron (a crosspolytope),
a quadrangular pyramid, a triangular prism or a cube.
\item The combinatorics of $D(G)$ coincides with the combinatorics
of the polytope $\mathrm{conv}\, D(G)$.
\item All quadrangular faces of $conv\, D(G)$ are parallelograms. 
\end{enumerate}
\end{proposition}

We proceed by recalling the definitions of the $\delta$- and $\pi$-surfaces
associated with a parallelohedron.

\begin{definition}
Let $P_\delta$, the {\em $\delta$-surface of $P$}, be the manifold 
obtained from $\partial P$, the surface of $P$, by removing all closed 
non-primitive $(d-2)$-faces. Also let the 
{\em $\pi$-surface of $P$} be defined by 
$P_\pi := P_{\delta} /(x \sim -x)$, i.e.,
$P_{\pi}$ is a quotient of $P_{\delta}$ obtained by identifying 
opposite points.
\end{definition}

\begin{definition}
Let $F$ and $G$ be two facets of $P$ such that the intersection 
$F\cap G$ is a primitive $(d - 2)$-face. Define the {\em gain function} 
$g(F,G)$ as follows. It is said that $g(F, G) = g$ if there exists
a triple $(a, b, c)$ of positive numbers such that
\begin{equation*}
g = \tfrac{b}{a} \qquad \text{and} \qquad
a\mathbf e_{P, P_F} +b \mathbf e_{P_G, P}+ c\mathbf e_{P_F, P_G}=0,
\end{equation*} 
where $P_F$ and $P_G$ are the tiles of $\mathcal T(P)$ that are adjacent
to $P$ by the facets $F$ and $G$ respectively, and $\mathbf e_{P_1, P_2}$
denotes the unit normal to a $(d - 1)$-face $P_1 \cap P_2$ directed from
$P_1$ towards $P_2$.
\end{definition}

\begin{remark}
There is a unique linear dependence between 
$\mathbf e_{P, P_F}$, $\mathbf e_{P_G, P}$ and  $\mathbf e_{P_F, P_G}$
up to a common positive multiplier,
moreover, the coefficients of the dependence have equal signs.  Hence 
$g(F, G)$ is defined in a consistent and unique way. 
\end{remark}

\begin{definition}
A continuous piecewise linear map $\gamma : [0, 1] \to \partial P$ 
is called a {\em path}
on $\partial P$. We call $\gamma$ a {\em generic path} if
\begin{itemize}
\item The points $\gamma(0)$ and $\gamma(1)$ are interior points of
some facets of $P$.
\item No point $\gamma(t)$, $t \in [0, 1]$, belongs to any of the 
$(d - 3)$-faces of $P$.
\item Whenever $\gamma$ intersects some $(d - 2)$-face of $P$,
the intersection is transversal.
\end{itemize}
If $\gamma([0, 1]) \subset P_{\delta}$, we call $\gamma$ a 
{\em primitive path}. 
\end{definition}

Given a generic primitive path $\gamma$ on $\partial P$, let 
\begin{equation*}
\langle \gamma \rangle := [F_0, F_1, \ldots, F_k], 
\end{equation*}
where $[F_0, F_1, \ldots, F_k]$ is the sequence of facets visited by 
$\gamma$. More precisely, the sequence $[F_0, F_1, \ldots, F_k]$ 
satisfies the following description: the path $\gamma$ starts on 
the facet $F_0$, then goes to $F_1$ crossing exactly one primitive face of codimension $2$, \ldots, and, finally, ends on $F_k$.
Define
\begin{equation*}
g(\gamma) := g(\langle \gamma \rangle) = g([F_0, F_1, \ldots, F_k]) := 
\prod\limits_{i = 1}^k g(F_i, F_{i - 1}).
\end{equation*}

Consider an arbitrary path (continuous piecewise linear map) 
$\gamma_{\pi} : [0, 1] \to P_{\pi}$. Since $P_{\delta}$ is a twofold
cover of $P_{\pi}$, there are two ways to lift $\gamma_{\pi}$ onto
$P_{\delta}$. If 
$\gamma_{\delta, 1}, \gamma_{\delta, 2}: [0, 1] \to P_{\delta}$ are the 
paths obtained by such a lift, then both $\gamma_{\delta, 1}$ and 
$\gamma_{\delta, 2}$ are primitive. If $\gamma_{\delta, 1}$ or
$\gamma_{\delta, 2}$ is generic, then so is the other.
In this case we call $\gamma_{\pi}$ {\em generic} as well. Moreover,
for generic $\gamma_{\pi}$ we have
\begin{equation}\label{eq:sympaths}
g(\gamma_{\delta, 1}) = g(\gamma_{\delta, 2}).
\end{equation}
Indeed, if $\langle \gamma_{\delta, 1} \rangle = [F_0, F_1, \ldots, F_k]$,
then $\langle \gamma_{\delta, 2} \rangle = [-F_0, -F_1, \ldots, -F_k]$, 
and the identity~\eqref{eq:sympaths} follows from the fact that
$g(F_i, F_{i - 1}) = g(-F_i, -F_{i - 1})$. Thus it is natural to put 
$g(\gamma_{\pi}) := g(\gamma_{\delta, 1})$. 

The following criterion holds.

\begin{proposition}[See~\cite{GGM}, Lemma~2.6 and Theorem~4.6]
\label{prop:unit_gain}
The following conditions are equivalent for a parallelohedron $P$.
\begin{enumerate}
\item $P$ is affinely Voronoi.
\item For every generic path $\gamma$ on $P_{\pi}$ which is closed, i.e., 
$\gamma(0) = \gamma(1)$, it holds that $g(\gamma)=1$.
\end{enumerate}
\end{proposition}

It will be useful to list some particular cases of closed curves
$\gamma$ on $P_{\pi}$ with $g(\gamma) = 1$.

\begin{lemma}\label{lem:basic-circuits}
Let $\gamma$ be a generic closed path on $P_{\pi}$. Assume that
$\gamma$ has a lift $\gamma_{\delta}$ onto $P_{\delta}$ satisfying
any of the conditions (HB), (TC) or (O) below. Then $g(\gamma)=1$.
\begin{enumerate}
\item[(HB)] $\langle \gamma_{\delta} \rangle = [F_1, F_2, F_3, -F_1]$,
where the facets $F_1$, $F_2$ and $F_3$ are parallel to some primitive
$(d - 2)$-face $G$ of $P$.
\item[(TC)] $\langle \gamma_{\delta} \rangle = [F_1, \ldots, F_k, F_1]$,
where all facets $F_1, \ldots, F_k$ are distinct and share a common
$(d - 3)$-face $G$ of $P$.
\item[(O)] $\langle \gamma_{\delta} \rangle = [F_1, F_2, F_3, F_1]$, where
\begin{equation*}
F_1 = P \cap (P + \mathbf x_2 - \mathbf x_3), \quad
F_2 = P \cap (P + \mathbf x_1), \quad
F_3 = P \cap (P + \mathbf x_2),
\end{equation*}
and $\{\mathbf 0, \mathbf x_1, \mathbf x_2, \mathbf x_3, 
\mathbf x_1 + \mathbf x_3 - \mathbf x_2 \}$ is the vertex set of a
pyramidal dual 3-cell $D(G)$ with apex $\mathbf 0$.
\end{enumerate}
\end{lemma}

Before we proceed with a proof, let us give a name to each type of closed
paths mentioned in Lemma~\ref{lem:basic-circuits}.

\begin{definition}
In the notation of Lemma~\ref{lem:basic-circuits}:
\begin{enumerate}
\item If the condition (HB) is satisfied, then $\gamma$ is called a
{\em half-belt circuit}.
\item If the condition (TC) is satisfied, then $\gamma$ is called a
{\em trivially contractible circuit}.
\item If the condition (O) is satisfied, then $\gamma$ is called an
{\em Ordine circuit}.
\end{enumerate}
\end{definition}

\begin{proof}[Proof of Lemma~\ref{lem:basic-circuits}]
Consider each condition separately.

\noindent {\bf Case (HB).} See~\cite[Lemma~4.5]{GGM}.

\noindent {\bf Case (TC).} See~\cite[Lemma~3.6]{GGM}.

\noindent {\bf Case (O).} This case follows from the existence of a local
canonical scaling around a $(d - 3)$-face $G$ whose dual is combinatorially
equivalent to a quadrangular pyramid. We provide the proof to make the 
argument self-contained.

Denote $\mathbf x_0 := \mathbf 0$ and 
$\mathbf x_4 := \mathbf x_1 + \mathbf x_3 - \mathbf x_2$. Then
\begin{equation*}
\{ P + \mathbf x_i : i = 0, 1, \ldots, 4 \}
\end{equation*}
is the set of all parallelohedra of $\mathcal T(P)$ incident to $G$.
By \cite[Lemma~3.7]{GGM}, there exist affine functions 
$U_i : \mathbb R^d \to \mathbb R$, $i = 0, 1, \ldots, 4$ such that if
$P + \mathbf x_i$ and $P + \mathbf x_j$ share a common facet $F_{ij}$
then $U_i$ and $U_j$ coincide on the affine hull of $F_{ij}$ and nowhere
else.
Define
\begin{equation*}
\mathbf a_{ij} := \mathrm{grad}\, U_j - \mathrm{grad}\, U_i.
\end{equation*}
Then the following identities hold.
\begin{equation}\label{eq:gradients}
g(F_1, F_2) = \frac{|\mathbf a_{01}|}{|\mathbf a_{14}|}, \quad
g(F_2, F_3) = \frac{|\mathbf a_{02}|}{|\mathbf a_{01}|}, \quad
g(F_3, F_1) = \frac{|\mathbf a_{23}|}{|\mathbf a_{02}|}.
\end{equation}
Let us prove, the first identity of~\eqref{eq:gradients}. 
The $(d - 2)$-face $F_1 \cap F_2$ is shared by exactly three 
parallelohedra of $\mathcal T(P)$, namely,
$P$, $P + \mathbf x_1$ and 
$P + \mathbf x_2 - \mathbf x_3 = P + \mathbf x_1 - \mathbf x_4$.
The facet $F_2$ is orthogonal to the vector $\mathbf a_{01}$. The facet
$F_1$ is parallel to the facet $(P + \mathbf x_1) \cap (P + \mathbf x_4)$ and therefore is orthogonal to the vector $\mathbf a_{14}$. The facet 
$(P + \mathbf x_1) \cap (P + \mathbf x_1 - \mathbf x_4)$
is parallel to the facet $(P + \mathbf x_4) \cap P$ and is therefore 
orthogonal to the vector $\mathbf a_{04}$. 
But $\mathbf a_{01} + \mathbf a_{14} - \mathbf a_{04} = \mathbf 0$,
hence indeed
$g(F_1, F_2) = \frac{|\mathbf a_{01}|}{|\mathbf a_{14}|}$. 

The proof of the third identity of~\eqref{eq:gradients} is obtained
from that of the first identity by interchange $\mathbf x_1$ with 
$\mathbf x_2$ and $\mathbf x_3$ with $\mathbf x_4$.

Concerning the second identity of~\eqref{eq:gradients}, 
the $(d - 2)$-face $F_2 \cap F_3$ is shared by parallelohedra 
$P$, $P + \mathbf x_1$ and $P + \mathbf x_2$. The normals to
the faces $F_2$, $F_3$ and $(P + \mathbf x_1) \cap (P + \mathbf x_2)$
are, respectively, $\mathbf a_{01}$, $\mathbf a_{02}$ and $\mathbf a_{12}$.
Since $\mathbf a_{01} + \mathbf a_{02} - \mathbf a_{12} = \mathbf 0$,
the second identity of~\eqref{eq:gradients} follows.

Finally, we have the identity
\begin{equation*}
\mathbf a_{12} + \mathbf a_{23} + \mathbf a_{34} - \mathbf a_{14}
= \mathbf 0.
\end{equation*} 
The vectors $\mathbf a_{12}$ and $\mathbf a_{23}$ span
a 2-dimensional space, and $\mathbf a_{34}$ and $\mathbf a_{14}$
are collinear to $\mathbf a_{12}$ and $\mathbf a_{23}$ respectively.
Hence $\mathbf a_{12} = -\mathbf a_{34}$ and 
$\mathbf a_{23} = -\mathbf a_{14}$. In particular, 
$|\mathbf a_{23}| = |\mathbf a_{14}|$.
 
Expanding the definition of $g(\gamma)$ via~\eqref{eq:gradients} yields
\begin{equation*}
g(\gamma) = g(F_1, F_2)g(F_2, F_3)g(F_3, F_1) = 
\frac{|\mathbf a_{01}|}{|\mathbf a_{14}|} \cdot
\frac{|\mathbf a_{02}|}{|\mathbf a_{01}|} \cdot
\frac{|\mathbf a_{23}|}{|\mathbf a_{02}|} = 
\frac{|\mathbf a_{23}|}{|\mathbf a_{14}|} 
= 1,
\end{equation*}
finishing the proof.
\end{proof}

To conclude this section, we reproduce the main result of~\cite{GGM}.
See Section~\ref{sec:graph} of this paper for further discussion
of the approach.

\begin{theorem}[\cite{GGM}, Theorem 4.6]\label{thm:ggm}
If the homology group $H_1(P_\pi, \mathbb Q)$ is generated by half-belt cycles, then $P$ is affinely Voronoi.
\end{theorem}

For the sake of brevity, we will call the condition of Theorem~\ref{thm:ggm}
{\it the GGM condition}.

\section{Simplicial complex approach}\label{sec:simplicial}

In this section we propose yet another sufficient condition for a parallelohedron to satisfy the Voronoi conjecture. Both GGM and Ordine
conditions are, apparently, special cases of our condition.

We will introduce the notion of a {\it Venkov complex} $Ven(P)$ 
associated with a parallelohedron $P$. By definition,
$Ven(P)$ will be a finite homogeneous 2-dimensional simplicial
complex. The name is justified by the observation that the edge structure
of $Ven(P)$ coincides with that of the {\it red Venkov graph} 
$VG_r(P)$. The graph $VG_r(P)$ may, however, have additional isolated 
vertices, and the number of isolated vertices 
is the number of 1-dimensional summands in the representation of $P$
as a direct sum of irreducible parallelohedra, i.e. those that can not be represented as direct sum of parallelohedra of smaller dimension..

Let $\mathcal A$ be an arbitrary set {\it (the alphabet)} of {\it labels}, 
$T_m(\mathcal A)$ be the set of all $m$-element subsets of $\mathcal A$.
Every finite subset $X \subseteq T_m(\mathcal A)$ defines a finite
homogeneous $(m - 1)$-dimensional simplicial complex $\mathcal C(X)$. 
Namely, the vertices of $\mathcal C(X)$ are in one-to-one correspondence with the set $\bigcup\limits_{S \in X} S$ 
(i.e., the set of labels that are used at least once). 
The facets of $\mathcal C(X)$ are in one-to-one correspondence with
elements of $X$ so that each $S \in X$ corresponds to a facet with
the vertex set labeled exactly by the elements of $S$.
For our purposes we set $\mathcal A := \Lambda(P) / 2\Lambda(P)$, 
i.e., the alphabet is the set of parity classes of the lattice $\Lambda(P)$.
The element $x + 2\Lambda(P) \in \Lambda(P) / 2\Lambda(P)$, i.e., the
parity class of the lattice point $x$, will be denoted by $\bar{x}$.

It will be convenient to use a shorthand notation 
\begin{multline*}
O(a, b, c, a', b', c') := \{ \{a, b, c\}, \{a', b', c'\}, \{a', b, c\},
\{a, b', c'\}, \{a, b', c\}, \\
\{a', b, c'\}, \{a, b, c'\}, \{a', b', c\} \}
\end{multline*}
One can see that $\mathcal C(O(a, b, c, a', b', c'))$ is combinatorially isomorphic to the surface of an octahedron with the pairs of opposite
vertices labeled as $\{a, a'\}$, $\{b, b'\}$ and $\{c, c'\}$. 

\begin{definition}
Let $P$ be a parallelohedron of dimension $d \geq 4$. Let $\mathcal D^3(P)$
denote the set of all dual 3-cells of the tiling $\mathcal T(P)$. 
For each $D \in \mathcal D^3(P)$ define a set
$X(D) \subseteq T_3(\mathcal A)$, where 
$\mathcal A := \Lambda(P) / 2\Lambda(P)$, as follows:
\begin{enumerate} 
\item If $D$ is a combinatorial tetrahedron and $V(D) = \{a, b, c, d\}$, set 
\begin{equation*}
X(D) := O\left( \overline{a + b}, \overline{a + c}, \overline{a + d},
\overline{c + d}, \overline{b + d}, \overline{b + c} \right).
\end{equation*}
\item If $D$ is a combinatorial pyramid, $V(D) = \{s, a, b, c, d\}$ and
$a + c = b + d$, set
\begin{equation*}
X(D) := O\left( \overline{s + a}, \overline{s + b}, 
\overline{a + d},
\overline{s + c}, \overline{s + d}, \overline{a + b}\right).
\end{equation*}
\item If $D$ is a combinatorial octahedron, $V(D) = \{a, b, c, d, e, f\}$
and $a + d = b + e = c + f$, set 
\begin{equation*}
X(D) := O\bigl( \overline{a + b} , 
\overline{a + c} , \overline{b + c} ,
\overline{a + e} , \overline{a + f} , \overline{b + f} \bigr).
\end{equation*}
\item If $D$ is a combinatorial prism, $V(D) = \{a, b, c, a', b', c'\}$
and $a - a' = b - b' = c - c'$, set
\begin{equation*}
X(D) := \{ \{ \overline{a + b} , \overline{a + c} , \overline{b + c} \} \}.
\end{equation*}
\item If $D$ is a combinatorial cube, set $X(D) := \varnothing$.
\end{enumerate}
Write, finally,
\begin{equation*}
X := \bigcup\limits_{D \in \mathcal D^3(P)} X(D).
\end{equation*}
Then the simplicial complex $Ven(P) := \mathcal C(X)$ is called the
{\it Venkov complex} of $P$. The faces of $Ven(P)$ are called the
{\it Venkov triangles}.
\end{definition}

\begin{remark}
The cases (2), (3), and (4) of dual 3-cells have certain linear relations between vertices. These relations force linear relations between parity classes as well.

For example in case (2), the relation $a+c=b+d$ implies that $\overline{a + d} = \overline{b + c}$ and $\overline{a + b} = \overline{c + d}$, so the set $X(D)$ can be written in an equivalent way as 
$$X(D) = O\left( \overline{s + a}, \overline{s + b}, 
\overline{b + c},
\overline{s + c}, \overline{s + d}, \overline{c + d} \right).$$

Similarly in case (3) , the relations $a+d=b+e=c+f$ between vertices of $D$ imply the following relations between parity classes: $\overline{a + b} = \overline{d + e}$, $ 
\overline{a + c} = \overline{d + f}$, $\overline{b + c} = \overline{e + f}$, $\overline{a + e} = \overline{b + d}$, $\overline{a + f} = \overline{c + d}$, $\overline{b + f} = \overline{c + e}$.

Finally in case (4), the relations $a - a' = b - b' = c - c'$ imply $\overline{a + b} = \overline{a' + b'}$, $\overline{a + c} = \overline{a' + c'}$, $\overline{b + c} = \overline{b' + c'}$.

In the sequel we may switch between equivalent parity classes without saying it explicitly.
\end{remark}

For further simplicity, we identify the vertices of $Ven(P)$ with their labels.

Let us recall the definition of Venkov graphs.

\begin{definition}[The Venkov graph, see, for 
instance,~\cite{Ordine_thesis})]
Let $P$ be a parallelohedron. Set
\begin{equation*}
V := \{ \{ F, -F \} \, \vert \, \text{$F$ is a facet of $P$} \}.
\end{equation*}
In other words, $V$ is the set of pairs of opposite facets of $P$. Let
now $\{ F, -F \}$ and $\{ F', -F' \}$ be two distinct elements of $V$.
We say that 
\begin{itemize}
\item $\bigl\{ \{ F, -F \}, \{ F', -F' \} \bigr\} \in E_b$ if 
$F \cap F'$ is a non-primitive $(d - 2)$-face of $\mathcal T(P)$. 
\item $\bigl\{ \{ F, -F \}, \{ F', -F' \} \bigr\} \in E_r$ 
if either $F \cap F'$ or $F \cap (-F')$ is a primitive $(d - 2)$-face of 
$\mathcal T(P)$.
\end{itemize}
Then $VG(P) := (V, E_b \cup E_r)$ is called
the {\it Venkov graph} of $P$ and $VG_r(P) := (V, E_r)$ 
(respectively,  $VG_b(P) := (V, E_b)$) is the {\it red}
(respectively, {\it blue}) {\it Venkov graph} of $P$. 
\end{definition}

The next definition establishes a correspondence between the Venkov complex
and the Venkov graph of a parallelohedron.

\begin{definition}\label{def:phi-map}
Given a parallelohedron $P$ of dimension $d \geq 4$, let a map
\begin{equation*}
\varphi : vert(Ven(P)) \to vert(VG_r(P))
\end{equation*}
be defined as follows. For each $x \in vert(Ven(P))$  (and thus
satisfying $x \in \Lambda(P) / 2\Lambda(P)$) we set
$\varphi(x) := \{ F, -F \}$ if $F = P \cap (P + a)$, where
$\overline{a} = \xi$.
\end{definition}

\begin{remark}
Each parity class of $\Lambda(P)$ (i.e., a coset in 
$\Lambda(P) / 2\Lambda(P)$) contains either two opposite facet
vectors of $P$, or does not contain facet vectors. From the definition of
$Ven(P)$ it follows immediately that each parity class 
$\xi \in vert(Ven(P))$ contains exactly two facet vectors, which, in turn,
define the pair $(F, -F)$ uniquely.
\end{remark}

\begin{lemma}
The map $\varphi$ from Definition~\ref{def:phi-map} has the following
properties.
\begin{enumerate}
\item $\varphi$ is injective.
\item $\varphi$ induces a bijection between the edge sets of $VG_r(P)$
and $Ven(P)$.
\end{enumerate}
\end{lemma}

\begin{proof}
{\bf Assertion (1)} holds since distinct parity classes of $\Lambda(P)$ 
determine distinct pairs of opposite facets.

\noindent {\bf Assertion (2)} is proved by verifying the properties 
{\bf (a)} and {\bf (b)} below.

\noindent {\bf (a)} If $\{ x, y \}$ is an edge of $Ven(P)$
then $\{ \varphi(x), \varphi(y) \}$ is an edge of $VG_r(P)$.
This property follows immediately from the definition of $Ven(P)$.

\noindent {\bf (b)} If $\bigl\{ \{ F, -F \}, \{ F', -F' \} \bigr\}$
is an edge of $VG_r(P)$, then $\varphi^{-1}(\{ F, -F \})$ and
$\varphi^{-1}(\{ F', -F' \})$ exist. Moreover, they are connected with
an edge of $Ven(P)$. In order to verify this property, assume, with
no loss of generality, that $F \cap F'$ is a primitive $(d - 2)$-face
of $P$. Then the property {\bf (b)} is immediate by considering
the set of triples $X(D(G))$, where $G$ is an arbitrary 
$(d - 3)$-subface of $F \cap F'$, and $D(G)$ is the dual 3-cell of $G$.
\end{proof}

By the following corollary, the Venkov complex is, in a sense,
a 2-dimensional extension of the red Venkov graph, which
justifies our terminology.

\begin{corollary}\label{cor:graph-is-1-skeleton}
Let $P$ be a parallelohedron of dimension $d \geq 4$. Then
red Venkov graph $VG_r(P)$ can be obtained by adding a finite number
(possibly, zero) of isolated vertices to the 1-dimensional skeleton
of $Ven(P)$.
\end{corollary}

\begin{remark}
The number of additional isolated vertices in $VG_r(P)$ equals 
the number of 1-dimensional summands in the representation of $P$
as a direct sum of irreducible parallelohedra, see~\cite{Ordine_thesis}.
\end{remark}

The next lemma explains our choice of the 2-dimensional face 
structure for the Venkov complex.

\begin{lemma}\label{lem:ven_face}
Let a triple $\{ x_1, x_2, x_3 \} \in T_3(\Lambda(P) / 2 \Lambda(P))$ 
span a 2-dimensional face of $Ven(P)$. For $i = 1, 2, 3$ let 
$\varphi^{-1}(x_i) = \{ F_i, -F_i \}$. Then
\begin{enumerate}
\item There exists a generic closed path $\gamma$ on 
$P_{\pi}$, whose lift $\gamma_{\delta}$ onto
$P_{\delta}$ satisfies 
\begin{equation*}
\langle \gamma_{\delta} \rangle = [\pm F_1, \pm F_2, \pm F_3, \pm F_1]
\end{equation*} 
for an appropriate choice of signs.
\item Such a path $\gamma$ is either half-belt, or trivially 
contractible, or Ordine.
\end{enumerate}
\end{lemma}

\begin{proof} {\bf Assertion (1).} There exists $\tilde{F_2} \in \{ F_2, -F_2 \}$
which is adjacent to $F_1$ by a primitive $(d - 2)$-face. Similarly,
choose $\tilde{F_3} \in \{ F_3, -F_3 \}$ adjacent to $\tilde{F_2}$ and 
$\tilde{F_1} \in \{ F_1, -F_1 \}$ adjacent to $\tilde{F_3}$. Then
it is possible to construct $\gamma_{\delta}$ so that 
$\langle \gamma_{\delta} \rangle = [F_1, \tilde{F_2}, 
\tilde{F_3}, \tilde{F_1}]$ and the image of $\gamma_{\delta}$ under
the natural projection $P_{\delta} \to P_{\pi}$ is a closed path $\gamma$.

\noindent {\bf Assertion (2).} By definition of $Ven(P)$, we have $\{ x_1, x_2, x_3 \} \in X(D)$, where
$D$ is a dual 3-cell for $\mathcal T(P)$. In particular 
$X(D) \neq \varnothing$, hence $D$ cannot be a combinatorial cube.
The rest follows from the definition of $Ven(P)$.

Indeed, if $D$ is a combinatorial tetrahedron or a combinatorial 
octahedron, then $\gamma$ is either a half-belt cycle, or a
trivially contractible cycle. If $D$ is a combinatorial quadrangular
pyramid, then $\gamma$ is either a half-belt cycle, or an
Ordine cycle. Finally, if $D$ is a combinatorial prism, then 
$\gamma$ is a half-belt cycle.
\end{proof}

We proceed with the main result of this section --- a sufficient condition 
for a parallelohedron to be affinely Voronoi in terms of its Venkov
complex. From now on we use the notation $C_k(K, R)$ (respectively, 
$C^k(K, R)$) for the spaces of chains (respectively, cochains) of a
simplicial complex $K$ with coefficients in a commutative ring $R$. 

\begin{theorem}\label{thm:simp-condition}
Let $P$ be a parallelohedron of dimension $d \geq 4$. If 
the first cohomology group $H^1(Ven(P), \mathbb R)$ is trivial, then
$P$ is affinely Voronoi.
\end{theorem}

\begin{proof}
Assume that $P$ is not affinely Voronoi. It suffices to construct 
a non-trivial cohomology class in $H^1(Ven(P), \mathbb R)$. Equivalently,
we will construct a cochain $c \in C^1(Ven(P), \mathbb R)$ such that 
the coboundary operator $\delta$ vanishes on $c$, but at the same time 
$c$ is not a coboundary itself (i.e., $c \neq \delta c'$ for any 
$c' \in C^0(Ven(P), \mathbb R)$).

Let $x_1, x_2 \in vert(Ven(P))$ be such that $\{ x_1,  x_2 \}$ is an edge 
of $Ven(P)$. For $i = 1, 2$ let $\{ F_i, -F_i \} = \varphi(x_i)$. Then 
$F_1$ is adjacent either to $F_2$, or to $-F_2$ by a primitive 
$(d - 2)$-face of $P$. We then set
\begin{equation}\label{eq:cochain_def}
\langle c, \overrightarrow{x_1 x_2} \rangle := \ln g(F_1, F_2) 
\quad \text{or} \quad
\langle c, \overrightarrow{x_1 x_2} \rangle := \ln g(F_1, -F_2),
\end{equation}
respectively.  

Let us prove that $c$ is a cocycle. 
Consider an arbitrary 2-dimensional face $\{ x_1, x_2, x_3 \}$ of $Ven(P)$.
Applying Lemma~\ref{lem:ven_face} to $\{ x_1, x_2, x_3 \}$, yields 
a generic closed path $\gamma$ on $P_{\pi}$. We have
\begin{equation*}
\langle \delta c, \{x_1, x_2, x_3 \} \rangle = 
\langle c, \partial \{x_1, x_2, x_3 \} \rangle = 
\langle c, \overrightarrow{x_1 x_2} \rangle +  
\langle c, \overrightarrow{x_2 x_3} \rangle + 
\langle c, \overrightarrow{x_3 x_1} \rangle = 
\ln g(\gamma) = 0,
\end{equation*}
where the last identity is a consequence of Lemma~\ref{lem:basic-circuits}.
Since $\delta c$ vanishes on every 2-face of $Ven(P)$, $c$ is indeed
a cocycle.
	
In turn, $c$ is not a coboundary. In order to prove that
we assume the converse, i.e., $c = \delta c'$, where $c'$ is a 
$0$-cochain. 

By Proposition~\ref{prop:unit_gain}, there exists a closed path
$\gamma$ on $P_{\pi}$ such that $g(\gamma) \neq 1$.
Let $\gamma_{\delta}$ be a lift of $\gamma$ onto $P_{\delta}$.
Suppose that $\langle \gamma_{\delta} \rangle = [F_1, F_2, \ldots, F_k]$.
Then
\begin{equation*}
[\{ F_1, -F_1 \}, \{ F_2, -F_2 \}, \ldots, \{ F_k, -F_k\} = \{ F_1, -F_1\} ]
\end{equation*}
is a cycle in $VG_r(P)$. Therefore $x_i := \varphi^{-1} (\{ F_i, -F_i \})$
exists for each $i = 1, 2, \ldots, k$, and $x_1 = x_k$. Consequently,
\begin{equation*}
\sum\limits_{i = 1}^{k - 1} \langle c, 
\overrightarrow{x_i x_{i + 1}} \rangle = 
\sum\limits_{i = 1}^{k - 1} (\langle c', x_{i + 1} \rangle - 
\langle c', x_i \rangle) = 0. 
\end{equation*}
On the other hand, 
\begin{equation*}
\sum\limits_{i = 1}^{k - 1} \langle c, 
\overrightarrow{x_i x_{i + 1}} \rangle = 
\sum\limits_{i = 1}^{k - 1} \ln g(F_i, F_{i + 1}) = \ln g(\gamma) \neq 0. 
\end{equation*}
A contradiction shows that indeed $c$ is not a coboundary. This concludes
the proof.
\end{proof}

\section{Graph approach}\label{sec:graph}

In this section we present a method to verify the 
GGM condition. This approach, using the group of cycles of the red 
Venkov graph $VG_r(P)$, was proposed in~\cite{Garber_4dim}.

The two definitions below enable us to use the relation between the topology of $P_{\pi}$ on one side, and the topology of $VG_r(P)$ and
$Ven(P)$ on the other side.

\begin{definition}\label{def:identification}
Let $\gamma_{\pi}$ be a generic closed path on $P_{\pi}$. Let 
$\gamma_{\pi}$ be lifted onto $P_{\delta}$ as $\gamma_{\delta}$, and let
\begin{equation*}
\langle \gamma_{\delta} \rangle = [F_1, F_2, \ldots, F_k].
\end{equation*}
Consider the two cases
\begin{enumerate}
\item Let $\gamma_{\delta}$ pass through at least two different
facets of $P$. In this case $x_i := \varphi^{-1}(\{ F_i, -F_i \})$ 
are well-defined for all $i = 1, 2, \ldots, k$. Then we say that 
$\gamma_{\pi}$ {\it is identified} with the cycles
\begin{equation*}
[\{ F_1, -F_1 \}, \{ F_2, -F_2 \}, \ldots, 
\{ F_k, -F_k \} = \{ F_1, -F_1 \}] \quad \text{and} \quad 
[x_1, x_2, \ldots, x_k = x_1]
\end{equation*}
of the red Venkov graph $VG_r(P)$ and the Venkov complex $Ven(P)$ 
respectively.
\item If $\gamma_{\delta}$ is contained in a single facet of $P$, we say
that $\gamma_{\pi}$ {\it is identified} with the empty cycle of $VG_r(P)$ 
(or $Ven(P)$).
\end{enumerate} 
\end{definition}

\begin{remark}
Conversely, consider an arbitrary cycle 
$[x_1, x_2, \ldots, x_{k - 1}, x_k = x_1]$,
where $k > 2$, $x_i \in vert(Ven(P))$ and $\overrightarrow{x_i x_{i + 1}}$
is an oriented edge of $Ven(P)$. Then there exists a closed generic 
path on $P_{\pi}$ identified with this cycle. 
\end{remark}

\begin{definition}
A cycle $c \in C_1(Ven(P))$ is called a {\it combinatorial half-belt cycle}
(respectively, {\it combinatorial trivially contractible cycle}\, or
{\it combinatorial Ordine cycle}) if it is identified with a half-belt
(respectively, trivially contractible or Ordine) cycle $\gamma$ on
$P_{\pi}$.
\end{definition}

Now we are ready to reformulate the GGM condition in terms of the Venkov
complex.

\begin{lemma}\label{lem:ggm-cond}
The following assertions are equivalent.
\begin{enumerate}
\item[(i)] The group $H_1(P_{\pi}, \mathbb Q)$ is generated by half-belt
cycles.
\item[(ii)] The implication 
$\bigl( A_1(c) \land A_2(c) \bigr) \Rightarrow B(c)$ holds 
for all cochains $c \in C^1(Ven(P), \mathbb Q)$, where
\begin{align*}
A_1(c) := & \quad \bigl[ \text{$\langle c, \gamma \rangle = 0$ for every 
half-belt cycle $\gamma \in C_1(Ven(P), \mathbb Q)$} \bigr], \\
A_2(c) := & \quad \bigl[ \text{$\langle c, \gamma \rangle = 0$ for every 
trivially contractible cycle 
$\gamma \in C_1(Ven(P), \mathbb Q)$} \bigr], \\ 
B(c) := & \quad \bigl[ \text{$c$ is a coboundary} \bigr].
\end{align*}
\end{enumerate}
\end{lemma}

\begin{remark}
The condition $B(c)$ is equivalent to the statement that 
$\langle c, \gamma \rangle = 0$ for every 1-cycle $\gamma$, i.e.,
for every $\gamma \in C_1(Ven(P), \mathbb Q)$ satisfying 
$\partial \gamma = 0$.
\end{remark}

\begin{proof}[Proof of Lemma~\ref{lem:ggm-cond}]
\noindent {\bf (i) $\Rightarrow$ (ii).} Let $P$ be a parallelohedron
$P$ satisfying (i). Consider an arbitraty cochain $c \in C^1(Ven(P), \mathbb Q)$ for which both $A_1(c)$ and $A_2(c)$ are true. We claim
that $B(c)$ is true as well.

Let $\gamma_{\pi} \subset P_{\pi}$ be a closed generic curve. By 
Definition~\ref{def:identification}, $\gamma_{\pi}$ is identified with
a cycle $\gamma \in C_1(Ven(P), \mathbb Q)$. We then set
\begin{equation*}
c^*(\gamma_{\pi}) := \langle c, \gamma \rangle.
\end{equation*}
Since $A_2(c)$ holds, the value of $c^*(\gamma_{\pi})$ depends only on
the homptopy type of $\gamma_{\pi}$. Therefore $c^*$ acts as a map
$c^* : \pi_1(P_{\pi}) \to \mathbb Q$. By construction $c^*$ is a
homomorphism, vanishes on the commutant of $\pi_1(P_{\pi})$ and
its image lies in the field $\mathbb Q$ of characteristic zero.  
Consequently, the action of $c^*$ on the group $H_1(P_{\pi}, \mathbb Q)$
is also well-defined. By the property $A_1(c)$, all half-belt cycles 
lie in the kernel of $c^*$. Using (i), we conclude that $c^*$
acts on $H_1(P_{\pi}, \mathbb Q)$ trivially, which is only possible
if $c$ is a coboundary. Hence the implication (i) $\Rightarrow$ (ii)
holds.



\noindent {\bf (ii) $\Rightarrow$ (i).} 
Assume that (i) is false for $P$.
Let $G$ be the proper subgroup of $H_1(P_\pi,\mathbb Q)$  generated by half-belt cycles. Let 
$\ell : H_1(P_\pi,\mathbb Q) \to \mathbb Q$ be a linear map
such that $G \subseteq Ker\, \ell \subsetneq H_1(P_\pi,\mathbb Q)$.

Denote 
\begin{equation*}
\bar{C}_1(Ven(P), \mathbb Q) := \{ \gamma \in C_1(Ven(P), \mathbb Q) :
\partial \gamma = 0 \}.
\end{equation*}
Let $\gamma \in \bar{C}_1(Ven(P), \mathbb Q)$.
Then there exists $n \in \mathbb N$ such that 
$n \gamma = \sum\limits_{i = 1}^s \gamma^i$, where
$\gamma^i = [x^i_1, x^i_2, \ldots, x^i_{k_i} = x^i_1]$.
For each $\gamma^i$ there exists a closed generic path $\gamma^i_{\pi}$
identified with it according to Definition~\ref{def:identification}. By
setting
\begin{equation*}
\langle \ell_*, \gamma \rangle := \frac{1}{n} 
\sum\limits_{i = 0}^s \langle \ell, h^i \rangle,
\end{equation*}
where $h_i$ is the homology class of $\gamma^i_{\pi}$, we obtain
a linear map $\ell_* : \bar{C}_1(Ven(P), \mathbb Q) \to \mathbb Q$.
Let $c \in C^1(Ven(P), \mathbb Q)$ be an arbitrary continuation of
$\ell_*$ onto $C_1(Ven(P), \mathbb Q)$. Clearly, $c$ satisfies conditions $A_1(c)$ and $A_2(c)$, but not $B(c)$.
\end{proof}

The second condition of the previous lemma can be stated in terms of the group of cycles of the red Venkov graph $VG_r(P)$. Recall that,
by Corollary~\ref{cor:graph-is-1-skeleton}, the edge structure of 
$VG_r(P)$ and the edge structure of the Venkov complex $Ven(P)$
are isomorphic.

\begin{definition}
A cycle of $VG_r(P)$ identified either with a half-belt or with a 
trivially contractible cycle on $P_{\pi}$ is called a {\em basic cycle}.
The set of all basic cycles is denoted by $\mathcal C(P)$.
\end{definition}

\begin{definition}
If we treat a finite (non-directed) graph $G$ as a one-dimensional simplicial complex, then the group $H_1(G,\QQ)$ is called {\em the group of cycles} of $G$.
\end{definition}

\begin{remark}
The group of cycles of $G$ is a free abelian group (or a linear space over $\QQ$) of rank $e-v+k$ where $e$ is the number of edges, $v$ is the number of vertices, and $k$ is the number of connected components of $G$.
\end{remark}

Then we can reformulate Lemma~\ref{lem:ggm-cond} in the following way.

\begin{lemma}\label{lem:cycles}
The group $H_1(P_\pi,\QQ)$ is generated by half-belt cycles if and only if the group of cycles of the red Venkov graph of $P$ is generated by $\mathcal{C}(P)$.
\end{lemma}

\begin{proof}
Since the second condition of Lemma \ref{lem:ggm-cond} does not use two-dimensional simplices of $Ven(P)$, we can substitute $Ven(P)$ with the red Venkov graph in it. Then the implication $\bigl( A_1(c) \land A_2(c) \bigr) \Rightarrow B(c)$ means that the rank of the subgroup generated by $\mathcal{C}(P)$ is equal to the rank of the group of cycles of $VG_r(P)$, and $\mathcal{C}(P)$ generates the group of cycles.
\end{proof}

\section{Computational results}\label{sec:computational}

This section describes the computer-assisted verification of the 
following two results.

\begin{theorem}\label{thm:comp-simp}
Let a 5-dimensional parallelohedron $P$ be equivalent to some
5-dimensional Voronoi parallelohedron. Then the cohomology
group $H^1(Ven(P), \mathbb R)$ is trivial.
\end{theorem}

\begin{theorem}\label{thm:comp-ggm}
Let a 5-dimensional parallelohedron $P$ be equivalent to some
5-dimensional Voronoi parallelohedron. Then the GGM condition holds
for $P$.
\end{theorem}

Theorem~\ref{thm:main} is an immediate corollary of each of the
above theorems, as explained below.

\begin{proof}[Proof of Theorem~\ref{thm:main}]
The ``only if'' part is straightforward.

The ``if'' part is a combination of either Theorems~\ref{thm:comp-simp}
and~\ref{thm:simp-condition}, or of Theorems~\ref{thm:comp-ggm} 
and~\ref{thm:ggm}. 
\end{proof}

It is sufficient to verify the conclusions of Theorems~\ref{thm:comp-simp}
and~\ref{thm:comp-ggm} for a single representative of each equivalence
class of 5-dimensional Voronoi parallelohedra.

The list of representatives is available due to the algorithm 
of~\cite{ClassifDim5}. Each representative $P_i$ 
($1 \leq i \leq 110\,244$) is presented in two equivalent ways:
\begin{enumerate}
\item As a cell of $\mathbf 0$ in the Voronoi tessellation for
$\mathbb Z^5$, where the metric is given by an explicit quadratic form
$Q_i$, i.e., $\| a \| = Q(a, a)$. $Q_i$ is presented by its $5 \times 5$
matrix with integer entries.
\item As a convex hull of a set of vertices given explicitly by listing
the coordinates. All coordinates are rational numbers. Additionally, every
face of $P_i$ is described by listing its vertices. 
\end{enumerate}

Our first goal is to compute the dual complex $\mathcal D(P_i)$, which is,
in this case, a Delaunay tessellation $\mathcal D(\mathbb Z^5, Q_i)$.
We use two different approaches. The direct approach uses 
the algorithm of~\cite{DSV_MathComp2009}, which is available 
in~\cite{polyhedral}. The second approach uses the following
proposition.

\begin{proposition}[See, for instance,~\cite{Ordine_thesis}]%
\label{prop:dual-faces}
The following assertions hold.
\begin{enumerate}
\item Let $G$ be a face of parallelohedron $P$. Then
\begin{equation*}
D(G) = \{ -v \, \vert \, \text{$v \in \Lambda(P)$ and $G + v$ is
a face of $P$} \}.
\end{equation*}
\item $D$ is a dual $k$-cell of $\mathcal D(P)$ if and only if 
$D = D(G) + v$, where $G$ is a $(d - k)$-face of $P$ and 
$v \in \Lambda(P)$. 
\end{enumerate}
\end{proposition}

Using Proposition~\ref{prop:dual-faces}, one can compute 
$\mathcal D(P_i)$ from the vertex presentation of $P_i$. 
All computations are performed over the field of rationals, therefore
we are not concerned about the issues with machine precision.

Now we proceed with Theorems~\ref{thm:comp-simp} and~\ref{thm:comp-ggm}
separately.

For Theorem~\ref{thm:comp-simp}, we use $\mathcal D(P_i)$ to construct
the simplicial complex $Ven(P_i)$. After that, we check the triviality 
of $H^1(Ven(P_i), \mathbb R)$ by verifying the identity
\begin{equation*}
\dim (Im\, \delta_0) = \dim (Ker\, \delta_1),
\end{equation*}
where $\delta_0$ and $\delta_1$ are restrictions of the coboundary operator
$\delta$ to the spaces $C^0(Ven(P_i), \mathbb R)$ and 
$C^1(Ven(P), \mathbb R)$, respectively.

But $\dim (Im\, \delta_0) = rank(\delta_0)$ and 
$\dim (Ker\, \delta_1) = f_1(Ven(P_i)) - rank(\delta_1)$.
Therefore the condition of Theorem~\ref{thm:simp-condition}
is equivalent to the identity
\begin{equation}\label{eq:cohom_ranks}
rank(\delta_0) + rank(\delta_1) - f_1(Ven(P_i)) = 0.
\end{equation}
The identity~\eqref{eq:cohom_ranks} is verified by passing
$Ven(P_i)$ to the {\tt GAP} package~\cite{GAP, simpcomp}, 
where the functions on the left-hand side 
of~\eqref{eq:cohom_ranks} are readily available. The scripts that process
the vertex representation of five-dimensional parallelohedra
into a {\tt GAP} program are available on the web-page~\cite{py-scripts}. 

Similarly, for Theorem~\ref{thm:comp-ggm}, the dual complex 
$\mathcal D(P_i)$ is used to construct the graph $VG_r(P_i)$
and to determine the half-belt and the trivially contractible
cycles. 

The group of cycles of $VG_r(P_i)$ has rank
$e - v + k$, where $e$ is the number of edges, $v$ is the number of
vertices and $k$ is the number of connected components.
Therefore, by Lemma~\ref{lem:cycles}, it is sufficient to verify
that the $\mathbb Q$-rank of the set $\mathcal C(P_i)$
equals $e - v + k$. This is done by a straightforward computation.

\section{Relations between sufficient conditions}\label{sec:reduction}

In this section we show that the cohomology condition generalizes both
Ordine's 3-irreducibility condition and the GGM condition. 

First we deal with the Ordine condition. Recall
that a parallellohedron $P$ is 3-irreducible if no
3-cell of the dual complex $\mathcal D(P)$ is equivalent to a prism
or a cube. Ordine~\cite{Ordine_thesis} proved that every
3-irreducible parallelohedron of dimension $\geq 5$ is affinely Voronoi.

\begin{lemma}\label{lem:cohom_from_ordine}
Let $P$ be a 3-irreducible parallelohedron of dimension $d \geq 5$. 
Then the group $H^1(Ven(P), \mathbb R)$ is trivial.  
\end{lemma}

\begin{proof}
Let $c \in C^1(Ven(P), \mathbb R)$ be a cocycle. We need to prove that
$c$ is a coboundary.

Let $F$ anf $F'$ be two facets of $P$ such that $F \cap F'$ is a 
primitive $(d - 2)$-face. Then there is an edge between 
$\{ F, -F \}$ and $\{ F', -F' \}$ in $VG_r(P)$. Consequently, 
$\overrightarrow{xx'}$, where $x := \varphi^{-1}(\{ F, -F \})$ and 
$x' := \varphi^{-1}(\{ F', -F' \})$, is an oriented edge of
$Ven(P)$.
 
Set
\begin{equation*}
g(F, F') := \exp \left( c(\overrightarrow{xx'}) \right).
\end{equation*}
One can check that plugging $g$ into the argument of~\cite{Ordine_thesis}
instead of the gain function for $P$ is sufficient to produce an analogue of a canonical scaling. More precisely, there exists a function 
$s$ mapping facets of $P$ to positive real numbers such that the 
identities
\begin{align*}
s(F') = s(F)g(F, F') & \quad 
\text{whenever $F \cap F'$ is a primitive $(d - 2)$-face,} \\
s(F) = s(-F) & \quad \text{for all facets $F$}
\end{align*}
are satisfied.

Letting $c'(x) := \ln s(F)$, where $F \in \varphi(x)$, 
yields $c = \delta c'$. Hence $c$ is indeed a coboundary.
\end{proof}

We proceed by considering the GGM condition.

\begin{definition}
A 2-dimensional face of the Venkov complex $Ven(P)$ is called a {\it Venkov
triangle}.
\end{definition}

\begin{lemma}\label{lem:ggm-cycles-via-venkov}
The following assertions hold.
\begin{enumerate}
\item[(i)] Each combinatorial half-belt cycle is a boundary of a 
Venkov triangle.
\item[(ii)] Each combinatorial trivially contractible cycle is an integer
combination of boundaries of Venkov triangles.  
\end{enumerate}
\end{lemma}

\begin{proof} 
Assertion (i).
Let $[x_1, x_2, x_3, x_1]$ be a combinatorial half-belt cycle. Then
one can choose two facets $F_1 \in \varphi(x_1)$ and $F_2 \in \varphi(x_2)$
so that $F_1 \cap F_2$ is a primitive $(d - 2)$-face of $P$. Choose
an arbitrary $(d - 3)$-face $G \subset F_1 \cap F_2$. If $D$ is
the dual 3-cell of $G$, then $\{ x_1, x_2, x_3 \} \in X(D)$, finishing
the proof of the assertion. 

Assertion (ii). Let $\gamma$ be a combinatorial trivially 
contractible cycle around a $(d - 3)$-face $G$. 
Let $D$ be the dual 3-cell of $G$. There are three possibilities 
for the combinatorial type of $D$
--- a tetraheron, a crosspolytope and a quadrangular pyramid.
If $D$ is a tetrahedron or a crosspolytope, then 
$\gamma = [x_1, x_2, x_3, x_1]$, where $\{ x_1, x_2, x_3 \} \in X(D)$. 
Hence $\gamma$ is itself a boundary of a Venkov triangle.
If $D$ is a quadrangular pyramid, then $\gamma = [x_1, x_2, x_3, x_4, x_1]$
and goes along an equator of the octahedron $\mathcal C(X(D))$.
Therefore $\gamma$ is a sum of boundaries of four Venkov triangles
comprising one of the hemispheres of $\mathcal C(X(D))$.
\end{proof}

\begin{lemma}
Let assertion (ii) of Lemma~\ref{lem:ggm-cond} hold. Then the group 
$H^1(Ven(P), \mathbb R)$ is trivial.  
\end{lemma}

\begin{proof}
Note that the groups $H^1(Ven(P), \mathbb R)$ and $H^1(Ven(P), \mathbb Q)$ are either both trivial, or both non-trivial. Therefore for the rest of
the proof we will be working over the field of rationals.

Let $c \in C^1(Ven(P), \mathbb R)$ be a 1-cocycle, i.e., 
$\delta c \equiv 0$. Equivalently, we have
\begin{equation*}
\langle c, \partial \tau \rangle = \langle \delta c, \tau \rangle = 0
\end{equation*}
for every Venkov triangle $\tau$.

By Lemma~\ref{lem:ggm-cycles-via-venkov}, every half-belt cycle and every
trivially contractible cycle can be represented as a combination of
boundaries of Venkov triangles.  Hence, in the notation of 
Lemma~\ref{lem:ggm-cond},  $A_1(c)$ and $A_2(c)$ hold. Therefore
$B(c)$ holds as well, i.e., $c$ is a coboundary.

We thus conclude that every cocycle in $C^1(Ven(P), \mathbb Q)$ is
a coboundary. Hence the group $H^1(Ven(P), \mathbb Q)$ is trivial, and 
so is the group $H^1(Ven(P), \mathbb R)$.
\end{proof}

\section{Concluding remarks}\label{sec:conclusion}

We conclude the paper by posing an open problem.

\begin{problem}\label{prob}
Determine whether the following statements are true or false.
\begin{enumerate}
\item The GGM condition holds for all Voronoi parallelohedra.
\item For every Voronoi parallelohedron $P$ of dimension $d \geq 5$
the cohomology group $H^1(Ven(P), \mathbb R)$ is trivial.
\end{enumerate}
\end{problem}

If any of the statements of Problem~\ref{prob} holds, then a hypothetical
counterexample to the Voronoi conjecture should be non-equivalent
to any Voronoi parallelohedron. 

In particular, for five-dimensional parallelohedra both these conditions are equivalent to the Voronoi conjecture, so if one wants to find a five-dimensional counterexample for the Voronoi conjecture, then it is enough to search for a parallelohedron that violates any of these conditions. So this poses one more open question.

\begin{problem}
Is every five-dimensional parallelohedron combinatorially Voronoi?	
\end{problem}

\end{document}